\documentclass[11pt]{article}

\usepackage{amsmath}
\usepackage{amssymb}
\usepackage{amsthm}
\usepackage{mathtools}
\usepackage{bm}
\usepackage{graphicx}
\usepackage{hyperref}
\newcommand{\ext}{\operatorname{ext}}

\usepackage[margin=1in]{geometry}
\usepackage{amsmath} 
\usepackage{amssymb} 
\usepackage{amsbsy} 
\usepackage{amsthm} 
\usepackage{amsfonts}
\usepackage{thmtools}
\usepackage{thm-restate}
\usepackage{bbm}
\usepackage{url}
\usepackage{booktabs}
\usepackage{mathrsfs}
\usepackage{multirow}
\usepackage{chngpage}
\usepackage{subcaption}
\usepackage{float}
\usepackage{subfiles}
\usepackage[math]{cellspace}
\usepackage[font={small}]{caption}
\usepackage{tikz, pgfplots}
\usepackage{algpseudocode}
\usepackage{graphicx}
\usepackage{algorithm}
\usepackage{fullpage}
\usepackage{chngcntr}
\usepackage[square,sort,comma,numbers]{natbib}
\usepackage{soul}
\usepackage{mathtools}
\usepackage{array}
\usepackage{arydshln}
\usepackage[letterspace=-85]{microtype}
\usepackage{nicefrac}
\usepackage{hyperref}

\usetikzlibrary{shapes.geometric,arrows,positioning,calc}
\usetikzlibrary{patterns}
\usetikzlibrary{decorations.pathmorphing,backgrounds,fit,matrix,decorations.pathreplacing}

\newcommand*{\pfstart}{\begin{proof}}
\newcommand*{\pfend}{\end{proof}}

\tikzstyle{vertex} = [circle, minimum size=0.1cm, inner sep=0pt, draw=black, fill=black]
\tikzstyle{circ} = [circle, minimum width=0.5mm, inner sep=0pt,draw,fill]
\tikzstyle{hcirc} = [circle,minimum width=5mm, inner sep=0pt,draw]
\tikzstyle{bcirc} = [circle, minimum width=1.5mm, inner sep=0pt,draw,fill]
\tikzstyle{bhcirc} = [circle, minimum width=1.5mm, inner sep=0pt,draw, dotted ]
\tikzstyle{ept} = [circle,minimum width=0mm, inner sep=0pt, white]
\tikzstyle{txt} = [text width=1.3cm,draw,rounded corners=3pt]
\tikzstyle{ncirc} = [circle,draw=black, inner sep=1pt]

\DeclarePairedDelimiter\iprod{\langle}{\rangle}

\algrenewcommand\algorithmicindent{1em}

\theoremstyle{remark}

\usepackage{bm}
\theoremstyle{remark}
\newtheorem*{claim*}{Claim}
\theoremstyle{remark}
\newtheorem*{remark*}{Remark}
\theoremstyle{remark}

\theoremstyle{plain}
\newtheorem{proposition}{Proposition}

\theoremstyle{plain}
\newtheorem{lemma}{Lemma}

\theoremstyle{definition}

\theoremstyle{definition}
\newtheorem*{assumption*}{Assumption}

\theoremstyle{definition}
\newtheorem*{example*}{Example}

\declaretheorem[name=Example, style=definition]{example}
\renewcommand\thmcontinues[1]{Cont.}

\theoremstyle{definition}
\newtheorem{definition}{Definition}

\theoremstyle{plain}
\newtheorem{theorem}{Theorem}

\theoremstyle{plain}
\newtheorem{corollary}{Corollary}

\graphicspath{{Figures/}}  

\usepackage{adjustbox}
\usepackage{natbib} 

\usepackage{adjustbox}
\usepackage{caption}

\newcommand{\TABLE}[3]{%
  \caption{#1}%
  \vspace{2pt}
  \noindent
  \begin{adjustbox}{max width=\textwidth}
  #2%
  \end{adjustbox}%
  \par\vspace{2pt}
  \noindent{\footnotesize #3}%
}

\theoremstyle{definition}

\usepackage{amsthm}
\usepackage{enumitem}

\hypersetup{
    colorlinks=true,
    citecolor=black,
    linkcolor=black,
    filecolor=magenta,      
    urlcolor=cyan,
}

\pgfplotsset{compat=1.15}
\tikzset{snake it/.style={decorate, decoration={snake, amplitude=.3mm, segment length=2mm}}}

\allowdisplaybreaks

\begin{document}

\title{Robust Network Enhancement: Formulations and Approximations}
	\author{
     Zolykha Rezaei
    \thanks{
    Department of Industrial, Manufacturing \& Systems Engineering, Texas Tech University, Lubbock, TX 79409 (email: {\tt zrezaei@ttu.edu}).}
    \and
    Ningji Wei
    \thanks{
    Department of Industrial, Manufacturing \& Systems Engineering, Texas Tech University, Lubbock, TX 79409 (email: {\tt ningji.wei@ttu.edu}). This work was partially supported by AFOSR grant FA9550-26-1-B345 and NSF grant 2548574.}
}
	\date{}
	\maketitle
\begin{abstract}
We study robust network enhancement problems that improve network performance through component upgrades, where both network utilities and enhancement costs are uncertain. We consider two formulations: a lower-bound problem that improves the minimum utility across feasible network structures and an upper-bound problem that improves the maximum utility. For the lower-bound problem, we develop an exact cut-generation reformulation and identify the robust separation subproblem as a key computational bottleneck. 
We characterize conditions under which fractional extreme points arise in the linearized robust formulation and demonstrate that such points can be optimal.
For the upper-bound problem, we derive two exact reformulations and compare the strength of their linear relaxations. To improve computational tractability, we introduce a box approximation of the utility uncertainty set that preserves solution robustness and provides explicit performance guarantees for both lower- and upper-bound problems. Computational experiments on spanning tree instances show that the proposed approximation improves computational efficiency for both problems while maintaining solution quality close to that of the exact formulations.
\par\vspace{0.7em}
\end{abstract}

\section{Introduction}

Infrastructure networks such as power grids, transportation networks,
communication systems, and supply chains play a critical role in supporting
economic and societal activities. The performance of these systems, however,
can deteriorate because of component failures, capacity limitations,
operational variability, or adversarial disruptions. To
maintain or improve network performance in a cost-effective manner, decision
makers may enhance a selected subset of existing components rather than
redesigning the entire network. These considerations have motivated network enhancement problems across a wide range of applications, including utility infrastructure
\citep{sabouhi2021electricity,wang2023collaborative}, public transportation
\citep{chen2023resilience,fan2011bilevel}, cybersecurity
\citep{wang2019cybersecurity}, and supply chains
\citep{shen2024multiobjective}. 

In many of these application settings, a common enhancement approach is to select a set of network components, such as vertices or edges, to improve infrastructure performance. For a given graph $\mathcal G=(V,E)$, we represent the infrastructure by a set of network structures, denoted by $\Omega$, which may include paths, flows, spanning trees, or other relevant subnetworks. For example, in power and water distribution networks, spanning trees are often used to represent feasible structures for delivering resources to the communities served by the network \citep{bartolin2005topological,mair2014spanning,sarkar2015kruskal}. Each component $e\in E$ has an additive utility $u_e\geq 0$, which can be increased by a fixed amount $\hat u_e$ at an enhancement cost $c_e$, thereby improving any feasible structure containing that edge. Since the minimum- and maximum-utility feasible structures characterize the network's attainable performance bounds, we focus on two corresponding enhancement problems.
Let $x\in\mathcal X$ indicate the components to enhance and $\mathcal Y_\Omega$ the indicator vectors of feasible structures. Given prescribed utility targets $r_{\mathrm{lb}}$ and $r_{\mathrm{ub}}$, the two problems are formulated as follows, where
$\langle \cdot,\cdot\rangle$ and $\cdot$ denote the inner and
componentwise products, respectively.

\noindent\begin{minipage}{0.45\textwidth}
{\setlength{\abovedisplayskip}{2pt}
\setlength{\abovedisplayshortskip}{2pt}
\begin{subequations} \label{eq:lb}
\begin{align}
  \min_{x\in \mathcal{X}} &~\langle c,x\rangle \\
  \text{s.t.} &
  \min_{y \in \mathcal Y_{\Omega}}
  \langle u + \hat u \cdot x, y \rangle
  \geq r_{\mathrm{lb}},
  \label{eq:lb01}
\end{align}
\end{subequations}
}
\end{minipage}
\hfill
\begin{minipage}{0.45\textwidth}
{\setlength{\abovedisplayskip}{2pt}
\setlength{\abovedisplayshortskip}{2pt}
\begin{subequations} \label{eq:ub}
\begin{align}
  \min_{x\in \mathcal{X}} &~\langle c,x\rangle \\
  \text{s.t.} &
  \max_{y \in \mathcal Y_{\Omega}}
  \langle u + \hat u \cdot x, y \rangle
  \geq r_{\mathrm{ub}}.
  \label{eq:ub01}
\end{align}
\end{subequations}
}
\end{minipage}
\vspace{6pt}

The lower-bound formulation~\eqref{eq:lb} seeks a minimum-cost enhancement plan such that the worst-case structure attains the prescribed utility target $r_{\mathrm{lb}}$. When $u$ is interpreted as the structure-selection cost incurred by a second decision maker with an objective opposing that of the first, this problem is equivalent to interdiction problems over network structures, such as paths \citep{israeli2002shortest,khachiyan2008shortpaths}, spanning trees \citep{bazgan2013criticaledges,wei2021integer}, cliques \citep{mahdavipajouh2020MinimumCostEdgeBlocker}, matchings \citep{zenklusen2010matching}, and vertex covers \citep{bazgan2011mostvitalnodes}.
The upper-bound formulation~\eqref{eq:ub}, in contrast, focuses on improving the best-case structure. Related ideas arise in network upgrade optimization. \citet{demgensky2004optimizing} study a network flow upgrade problem in which investments can either decrease arc costs or increase capacities, to maximize network flow subject to budget and flow-cost constraints. \citet{saharoy2015approximation} study a budget-constrained spanning tree problem in which edge weights can be increased at a cost, with the objective of maximizing the weight of the resulting maximum spanning tree.

Most existing studies on formulations~\eqref{eq:lb} and~\eqref{eq:ub} assume deterministic enhancement costs and edge utilities. In practice, however, these quantities may vary due to changes in material and labor costs, demand, operating conditions, or disruptions \citep{benAmmar2015commonrisk,zhang2016greensupplychain,dou2024mitigatingclosedloop}. To hedge against these uncertainties, this paper focuses on robust counterparts of the two enhancement problems, where the enhancement cost and utility vectors are assumed to belong to nonempty, compact, polyhedral uncertainty sets \(\mathcal C\) and \(\mathcal U\), respectively.
The resulting robust counterparts are given by
\vspace{6pt}
\noindent\begin{minipage}
{0.45\textwidth}
{\setlength{\abovedisplayskip}{2pt}
\setlength{\abovedisplayshortskip}{2pt}
\begin{subequations} \label{eq:rlb}
\begin{align}
  \min_{x\in \mathcal{X}}
  &\max_{c \in \mathcal C}~\langle c,x\rangle
  \label{eq:rlb01}\\
  \text{s.t.} &
  \min_{y \in \mathcal Y_{\Omega}}
  \min_{u \in \mathcal U}
  \langle u + \hat u \cdot x, y\rangle
  \geq r_{\mathrm{lb}},
  \label{eq:rlb02}
\end{align}
\end{subequations}
}
\end{minipage}
\hfill
\begin{minipage}{0.45\textwidth}
{\setlength{\abovedisplayskip}{2pt}
\setlength{\abovedisplayshortskip}{2pt}
\begin{subequations} \label{eq:rub}
\begin{align}
  \min_{x\in \mathcal{X}}
  &\max_{c \in \mathcal C}~\langle c,x\rangle \\
  \text{s.t.} &
  \max_{y \in \mathcal Y_{\Omega}}
  \min_{u \in \mathcal U}
  \langle u + \hat u \cdot x, y\rangle
  \geq r_{\mathrm{ub}}.
  \label{eq:rub02}
\end{align}
\end{subequations}
}
\end{minipage}

\noindent In both formulations, the two sources of uncertainty are modeled adversarially from the network owner's perspective: enhancement costs are overestimated, while utilities are underestimated, to hedge against potential parameter perturbations prescribed by $\mathcal C$ and $\mathcal U$.

Network interdiction and upgrade problems are often computationally challenging even in deterministic settings, and incorporating uncertainty introduces an additional optimization layer, leading to the multilevel combinatorial robust problems~\eqref{eq:rlb} and~\eqref{eq:rub}. Consequently, these formulations have received relatively limited attention despite their practical relevance. In the maximum-flow interdiction setting, \citet{chauhan2024robust} introduce a budget-constrained version of the robust lower-bound problem. Owing to the computational difficulty of the resulting robust multilevel formulation, their solution approach focuses instead on a more tractable upper-bounding problem. To the authors' best knowledge, this represents the most recent development directly related to~\eqref{eq:rlb}. A related line of work studies network interdiction under incomplete and asymmetric information \citep{azizi2024shortest,punlagreen2024shortest}. 
In these models, the uncertainty faced by the outer and inner decision makers is separated according to their asymmetric information, resulting in uncertainty structures and solution approaches that differ from those considered in this paper. 
Thus, despite related progress, exact and approximation methods for robust lower- and upper-bound enhancement problems remain limited.
This paper makes the following contributions.

\begin{itemize}
\item 
We develop exact formulations and solution methods for both robust enhancement problems. For the lower-bound problem~\eqref{eq:rlb}, we derive a cut-generation reformulation. For the upper-bound problem~\eqref{eq:rub}, we establish two exact reformulations and compare their linear relaxation strength.

\item 
For the robust lower-bound problem~\eqref{eq:rlb}, we identify the robust separation subproblem as a key computational bottleneck. Even when the underlying network problem retains an exact continuous relaxation, its linearized robust formulation can admit fractional optimal extreme points. We characterize the emergence of these fractional extreme points and demonstrate their computational implications (Proposition~\ref{prop:yfrac}).

\item 
To address these computational challenges, we develop a box approximation of the utility uncertainty set that simplifies both robust enhancement problems while preserving robustness. We establish explicit performance guarantees, showing that the approximation is no more conservative than increasing the enhancement target by a factor of $1+\lambda_{\mathcal U}$. Moreover, $\lambda_{\mathcal U}$ is bounded by the maximum relative utility perturbation, providing a dimension-independent guarantee.



\item Computational experiments on spanning tree enhancement instances show that the proposed approximation maintains consistently strong solution quality while providing substantial improvements in computational efficiency, especially in the lower-bound problem.
\end{itemize}

\subsection{Literature Review}
\label{sec:literature}
We organize the relevant literature into three streams. We begin with network interdiction and robust optimization, which provide the methodological foundations for our study, followed by robust lower- and upper-bound optimization.

\subsubsection{Interdiction and Robust Optimization}

Network interdiction is closely related to the deterministic lower-bound problem~\eqref{eq:lb}. These problems model adversarial interactions between an attacker and a defender, where the attacker first removes, disables, or degrades selected network components, and the defender subsequently optimizes the disrupted network. Interdiction has been studied across various network structures, including maximum flow \citep{wood1993deterministic}, shortest path \citep{israeli2002shortest}, multicommodity flow \citep{lim2007algorithms}, spanning trees \citep{bazgan2013criticaledges,wei2021integer}, cliques \citep{mahdavipajouh2020MinimumCostEdgeBlocker}, and matchings \citep{zenklusen2010matching}. These models typically seek either to maximize disruption under a fixed interdiction budget or to minimize the interdiction cost required to achieve a prescribed performance target \citep{wei2022integer,wei2025supervalid}. Their sequential decision structure naturally leads to bilevel or multilevel formulations, motivating solution approaches based on dualization, decomposition, constraint generation, and branch-and-cut \citep{smith2020survey,wei2022integer,wei2025supervalid}.

Robust optimization addresses parameter uncertainty by optimizing against adverse realizations within prescribed uncertainty sets \citep{ben2002robust,bertsimas2003robust}. It has been widely applied to inventory management \citep{bertsimas2006robust,Ardestani-JaafariDelage2016}, supply chain optimization \citep{STZ2019,bandi2019sustainable}, scheduling and transportation \citep{ardestani2021linearized}, and other decision problems. A central consideration is the tradeoff between robustness and computational tractability, which depends strongly on the structure of the uncertainty set. In particular, box uncertainty sets often lead to tractable reformulations, while more general uncertainty sets can introduce additional computational challenges \citep{ben2004adjustable,MarandiHertog18}. To address such challenges, several studies investigate tractable approximations and establish performance guarantees that quantify the resulting conservatism \citep{bertsimas2011geometric,bertsimas2015tight,awasthi2015adaptivity,wei2024adjustability}. 
These studies highlight the role of uncertainty-set structure in balancing computational tractability and conservatism, a central consideration in our robust network enhancement problems.

\subsubsection{Robust Lower-Bound Optimization}
Robust network interdiction is computationally challenging because uncertainty introduces an additional optimization layer into an already nested problem, often resulting in multilevel formulations. For the problem most closely related to our robust lower-bound model~\eqref{eq:rlb}, \citet{chauhan2024robust} derive a single-level mixed-integer programming formulation that provides an upper bound rather than solving the original robust bilevel formulation. They further develop Lagrangian-relaxation and Benders-decomposition heuristics to enhance solution efficiency.


Another related stream considers network interdiction under incomplete or
asymmetric information. \citet{azizi2024shortest} study shortest path interdiction in which the defender observes the interdicted arcs but faces
uncertainty about the interdiction intensity and the resulting traversal
costs. \citet{punlagreen2024shortest}, in contrast, assume that the defender
uses nominal network parameters while the attacker accounts for parameter
uncertainty. Thus, uncertainty is treated asymmetrically across the two decision levels in these models, which can lead to more tractable formulations and differs from the robust lower-bound formulation~\eqref{eq:rlb} considered in this paper.

Uncertainty in interdiction has also been studied through stochastic and
distributionally robust settings. Early stochastic formulations include
\citet{cormican1998stochastic}, who consider uncertain interdiction outcomes
and arc capacities. Subsequent stochastic interdiction models address uncertainty in network
parameters, interdiction outcomes, and asymmetric information under different
risk measures
\citep{lei2018stochastic,nguyen2022network,nguyen2024asymmetric,
pay2019stochastic}. Distributionally robust extensions further account for
uncertainty in the underlying probability distribution.
\citet{kang2023distributionally} develop risk-averse and risk-receptive
interdiction models over ambiguity sets, while
\citet{sadana2023randomized} study distributionally robust maximum-flow
interdiction with uncertain capacities, conditional value-at-risk, and
randomized interdiction strategies.

\subsubsection{Robust Upper-Bound Optimization}

Compared with the lower-bound formulation, the upper-bound enhancement problem~\eqref{eq:ub} has received considerably less attention. The closest related studies consider deterministic network upgrades. \citet{krumke1999improving} study node upgrades that reduce incident edge weights to improve the minimum spanning tree. \citet{saharoy2015approximation} investigate a budget-constrained problem that maximizes the weight of a maximum spanning tree through edge-weight upgrades. \citet{demgensky2004optimizing} consider network flow upgrades that reduce arc costs or increase capacities under a fixed budget. Unlike our formulation~\eqref{eq:rub}, these studies optimize network performance under deterministic parameters. To the best of our knowledge, robust upper-bound enhancement with uncertain costs and utilities remains largely unexplored in general network settings.



The remainder of the paper is organized as follows. Section~\ref{Lower_bound} develops an exact reformulation for the robust lower-bound problem and examines the computational difficulty of the associated robust separation problem. Section~\ref{sec:Lower_appx} proposes an approximation approach with performance guarantees. Section~\ref{section:robut_upper_bound} develops two exact reformulations for the robust upper-bound problem, compares their relaxation strength, and extends the approximation approach to this setting. Section~\ref{Section Numerical Exper} presents computational experiments evaluating the exact and approximate methods for both problems. Finally, Section~\ref{sec:conclusion} summarizes the main findings and concludes the paper.


\paragraph{Notation.}
For vectors $a$ and $b$, $\langle a,b\rangle$ denotes their inner product, and $a\cdot b$ denotes their componentwise product. Vector inequalities are interpreted componentwise, with scalar bounds applying to every component. For a polyhedron $\mathcal P$, $\operatorname{ext}(\mathcal P)$ denotes its set of extreme points. 
Given a graph $\mathcal G=(V,E)$, let $n:=|V|$ and $m:=|E|$ denote the numbers of vertices and edges, respectively. The cost and utility uncertainty sets are polytopes defined by $\mathcal C=\{c\in\mathbb R_+^m\mid Gc\le g\}$ and $\mathcal U=\{u\in\mathbb R_+^m\mid Hu\le h\}$, respectively.

\section{Robust Lower-Bound Enhancement} \label{Lower_bound}
This section studies the robust lower-bound problem~\eqref{eq:rlb}. We derive an equivalent reformulation by identifying a subset of structures that must be enhanced, leading to a cut-generation algorithm tailored to the robust setting. 
We further show that the associated robust separation subproblem can be computationally challenging even when its deterministic counterpart admits an efficient solution method.



\subsection{Reformulation via Critical Structures}

We begin by introducing the structures needed to reformulate the robust lower-bound problem~\eqref{eq:rlb}.

\begin{definition}[Critical Structure]
A structure $T\in\Omega$ is called \emph{critical} if it satisfies
$\min_{u\in\mathcal U}\langle u,y_T\rangle<r_{lb}$. The set of critical structures is denoted by $\hat{\Omega}
:=
\left\{
T\in\Omega
\;\middle|\;
\min_{u\in\mathcal U}\langle u,y_T\rangle<r_{lb}
\right\}.
$
\end{definition}

The notion of a critical structure has been employed in several interdiction problems to derive cut-generation reformulations \citep{wei2022integer,israeli2002shortest,lozano2017backward}. 
In our setting, we extend this definition to accommodate the uncertainty in utility.
Building on this definition, we establish the following reformulation for the robust lower-bound problem.
\begin{theorem}\label{Theo:Lb_master}
  Let $\hat\Omega$ denote the set of critical structures. Then formulation~\eqref{eq:rlb} is equivalent to
\setlength{\abovedisplayskip}{2pt}
\setlength{\abovedisplayshortskip}{2pt}
\setlength{\belowdisplayskip}{2pt}
\setlength{\belowdisplayshortskip}{2pt}
  \begin{subequations}
    \label{eq:rlbreform}
  \begin{align}
    \min_{\eta \geq 0, x \in \mathcal X} &~ \langle g, \eta \rangle\label{eq:rlbreform00}\\
    \text{s.t.} &~ G^\intercal \eta \geq x\label{eq:rlbreform01}\\
                &~ \langle u + \hat u \cdot x, y_T \rangle \geq r_{lb}, ~\forall T \in \hat\Omega, ~ u \in \mathcal U. \label{eq:rlbreform02}
  \end{align}
  \end{subequations}
\end{theorem}
\begin{proof}
Constraints~\eqref{eq:rlbreform00} and~\eqref{eq:rlbreform01} are obtained from~\eqref{eq:rlb01} by fixing the decision vector $x$ and dualizing the resulting inner problem with
    $\eta$ denoting the corresponding dual variables.
Constraint~\eqref{eq:rlbreform02} is equivalent to~\eqref{eq:rlb02}, since the latter holds if and only if every critical structure is strengthened such that its worst-case utility is at least $r_{lb}$. This yields the reformulation~\eqref{eq:rlbreform}.
\end{proof}

Theorem~\ref{Theo:Lb_master} reformulates the original multilevel optimization problem~\eqref{eq:rlb} as an equivalent single-level problem with a semi-infinite constraint. We therefore employ a cut-generation framework based on a master--subproblem decomposition. At each iteration, the master problem solves a relaxation of~\eqref{eq:rlbreform} containing only constraints indexed by a finite set of \(T \in \hat{\Omega}\) and their associated worst-case \(u \in \mathcal{U}\), generated in previous iterations. Given the resulting incumbent enhancement decision \(x\), the corresponding separation subproblem searches for a violated constraint associated with a critical structure \(T \in \hat{\Omega}\). Any identified violated constraint is then added to the master problem, and the process is repeated. 
The algorithm terminates when the separation subproblem identifies no violated constraint, at which point the current master solution is feasible for the original robust problem and therefore optimal.
Consequently, the overall computational efficiency of the framework depends primarily on the tractability of the separation subproblem, which we examine in the next subsection.


\subsection{Robust Separation Subproblem}

Given a solution $ \hat x\in\mathcal{X}$ obtained from the master problem, we consider the following subproblem, which we refer to as the \emph{robust bilinear separation subproblem}:
\begin{equation}
  \label{eq:splb}
  \begin{aligned}
    \min_{y \in \mathcal Y_\Omega,\, u \in \mathcal U}
     & ~\langle u + \hat u \cdot \hat x, y \rangle.
  \end{aligned}
\end{equation}
In many network optimization settings (e.g., shortest paths, minimum spanning trees, and maximum flow), the binary solution space $\mathcal Y_\Omega$ admits a \emph{positive perfect relaxation} $\bar {\mathcal Y}_\Omega$~\citep{conforti2014perfectFormulations}. That is, for every $u \geq 0$, the problem $\min_{y\in \bar {\mathcal{Y}}_{\Omega}}\langle u, y \rangle$ possesses a binary optimal extreme point.
A natural question is whether this property is preserved in the robust bilinear counterpart~\eqref{eq:splb}. The following proposition gives a positive answer.

\begin{proposition}
  \label{prop:relax}
  Suppose $\mathcal Y_\Omega$ admits a positive perfect relaxation $\bar{\mathcal Y}_\Omega$. Then, for every $x \in \mathcal X$, 
\begin{equation*}
    \begin{aligned}
        \min_{y\in \mathcal Y_\Omega}\min_{u \in \mathcal U} ~\langle u + \hat u \cdot x, y \rangle = \min_{y\in \bar{\mathcal Y}_\Omega}\min_{u \in \mathcal U} ~\langle u + \hat u \cdot x, y \rangle.
    \end{aligned}
\end{equation*}
\end{proposition}
\begin{proof}
Fix any $x\in\mathcal X$. For every $u\in\mathcal U$, the vector $u+\hat u\cdot x$ is nonnegative. Since $\bar{\mathcal Y}_\Omega$ is a positive perfect relaxation of $\mathcal Y_\Omega$, it follows that
\[
\min_{y\in\mathcal Y_\Omega}\langle u+\hat u\cdot x,y\rangle
=
\min_{y\in\bar{\mathcal Y}_\Omega}\langle u+\hat u\cdot x,y\rangle,
\qquad \forall u\in\mathcal U.
\]
This implies that minimizing over $u\in\mathcal U$ on both sides preserves the equality.
\end{proof}

Although Proposition~\ref{prop:relax} allows us to relax $\mathcal Y_\Omega$ in the subproblem, the resulting problem remains challenging due to the bilinear term in the objective. The following proposition reformulates~\eqref{eq:splb} as a mixed-integer program (MIP). We omit the proof, as the result follows from a standard MIP reformulation.

\begin{proposition}
  The problem \eqref{eq:splb} can be equivalently reformulated as the following \emph{linearized MIP separation subproblem}
 \setlength{\abovedisplayskip}{2pt}
\setlength{\abovedisplayshortskip}{2pt}
\setlength{\belowdisplayskip}{2pt}
\setlength{\belowdisplayshortskip}{2pt}
  \begin{subequations}
    \label{eq:spmip}
  \begin{align}
    \min_{y \in \mathcal Y_\Omega, u \in \mathcal U, z \geq 0} &~ \langle 1, z \rangle \label{eq:spmip00}\\
    \text{s.t.} &~ z \geq (u + \hat u \cdot \hat x) - (\bar u + \hat u\cdot \hat x)\cdot (1- y),\label{eq:spmip01}
  \end{align}
  \end{subequations}
where $\bar u_e = \max_{u \in \mathcal U} u_e$ is the entry-wise upper bound.
\end{proposition}

Following this linearization, 
it remains to determine whether the relaxation result established in Proposition~\ref{prop:relax} extends to the linearized MIP separation subproblem~\eqref{eq:spmip}.
The following proposition and subsequent example answer this question in the negative. The proposition provides an exact characterization of when the positive perfect relaxation may fail in this linearized formulation, while the subsequent example provides an explicit construction demonstrating this failure.


\begin{proposition}\label{prop:yfrac}
  Suppose $\bar u_e > 0$ for all $e \in E$, and let $\Pi$ be the relaxed feasible region of \eqref{eq:spmip} defined as $\Pi:=\{y \in \bar{\mathcal Y}_\Omega, u \in \mathcal U, z \geq 0 \mid (y, u, z) \text{ satisfies}~\eqref{eq:spmip01} \}$. 
  Let $y' \in \bar{\mathcal Y}_\Omega \setminus \operatorname{ext}(\bar{\mathcal Y}_\Omega)$,
 $u'\in\ext(\mathcal U)$,
and $z'_e := \max\{0, u'_e + \hat u_e x_e - (\bar u_e + \hat u_e x_e)(1-y'_e)\}$.
Define
\[
K(y',u')
:=
\left\{
e\in E\;\middle|\;
u'_e+\hat u_e x_e
=
(\bar u_e+\hat u_e x_e)(1-y'_e)
\right\},
\]
and let $D(y')$ denote the set of all feasible moving directions at $y'$, i.e., 
\[
D(y') := \left\{ d \in \mathbb{R}^{E} \,\middle|\, y' \pm \lambda d \in \bar{\mathcal Y}_\Omega \text{ for some } \lambda > 0 \right\}.
\]
Then,
\(
(y',u',z')\in\ext(\Pi)
\)
if and only if, for every nonzero $d\in D(y')$, there exists
$e\in K(y',u')$ such that $d_e\neq 0$.
\end{proposition}

\begin{proof}
  Let
\(
g_e(y):=
u'_e+\hat u_e x_e
-(\bar u_e+\hat u_e x_e)(1-y_e),
\)
so that $z'_e=\max\{0,g_e(y')\}$. Since $\bar u_e+\hat u_e x_e>0$, 
$g_e(y)$ is affine and strictly increasing in $y_e$.
For necessity, suppose $(y',u',z')\in\ext(\Pi)$ and, to the contrary, there exists a nonzero $d\in D(y')$ such that $d_e=0$ for all $e\in K(y',u')$. For sufficiently small $\lambda>0$, let
\[
y^\pm:=y'\pm\lambda d\in\bar{\mathcal Y}_\Omega,
\qquad
z^\pm_e:=\max\{0,g_e(y^\pm)\}.
\]
For $e\in K(y',u')$, we have $d_e=0$, and hence
$z^\pm_e=z'_e=0$. For $e\notin K(y',u')$, $g_e(y')\neq0$; thus,
for sufficiently small $\lambda$, $g_e(y^+)$ and $g_e(y^-)$ have the
same sign as $g_e(y')$. Therefore,
\(
(z^+_e+z^-_e)/2=z'_e
\).
Hence
\(
(y',u',z')
=
0.5 (y^+,u',z^+)
+
0.5 (y^-,u',z^-),
\)
where the two points are distinct and belong to $\Pi$, contradicting
$(y',u',z')\in\ext(\Pi)$.

For sufficiency, suppose every nonzero $d\in D(y')$ has
$d_e\neq0$ for some $e\in K(y',u')$. Assume
\(
(y',u',z')
=
0.5(y^1,u^1,z^1)
+
0.5(y^2,u^2,z^2)
\)
for two points in $\Pi$. Since $u'\in\ext(\mathcal U)$,
we must have $u^1=u^2=u'$. Let $d:=y^1-y'=y'-y^2$, so $d\in D(y')$.
For every $e\in K(y',u')$, $z'_e=0$ by definition. Since $z^1_e,z^2_e\ge0$ and
$(z^1_e+z^2_e)/2=z'_e$, it follows that $z^1_e=z^2_e=0$, which implies
\(
g_e(y^1)\le0\) and \(g_e(y^2)\le0
\).
Since $g_e(y')=0$ and $y'=(y^1+y^2)/2$, both inequalities must hold
with equality; that is $g_e(y^1) = g_e(y^2) = 0$. As $\bar u_e+\hat u_e x_e>0$, this implies
$d_e=0$ for every $e\in K(y',u')$. By assumption, this is possible
only if $d=0$, so $y^1=y^2=y'$.
Finally, since $z'$ is the componentwise minimum feasible $z$ for
$(y',u')$, the equality $z'=(z^1+z^2)/2$ implies
$z^1=z^2=z'$. Thus the two points coincide with $(y',u',z')$, proving
that $(y',u',z')\in\ext(\Pi)$.
\end{proof}

This proposition shows that, when $u \in \ext(\mathcal U)$, the additional extreme points introduced by the linearized formulation are driven by the big-$M$ terms $\bar u_e+\hat u_e x_e$ through the set $K(y,u)$. More generally, fractional extreme points may also arise from the joint interaction between $y$ and a non-extreme $u$. The following example demonstrates that the continuous relaxation of~\eqref{eq:spmip} can indeed admit a fractional optimal extreme point.

\begin{example}
Consider the three-edge triangle network, where
the spanning-tree polytope $\bar{\mathcal Y}_\Omega$ and the utility uncertainty $\mathcal U$ are defined as
\[
\bar{\mathcal Y}_\Omega
=\{y\in[0,1]^3\mid \langle 1,y\rangle=2\}, \quad 
\mathcal U
=\left\{u\in[2,10]^3\;\middle|\;
\langle 1,10-u\rangle\le20\right\}.
\]
Setting $x=0$ and $\hat u_e>0$ for all $e\in E$,
the relaxed linearized separation problem becomes
\[
\min_{y\in\bar{\mathcal Y}_\Omega,\,
u\in\mathcal U,\,z\ge0}
\left\{\langle 1,z\rangle
\mid z\ge u-10+10y\right\}.
\]
The point $y^*=(4/5,4/5,2/5)$, $u^*=(2,2,6)$, $z^*=0$ is feasible and attains the minimum objective value $0$. Moreover, $u^*\in\operatorname{ext}(\mathcal U)$ and $K(y^*,u^*)=E$. Since $y^*$ is fractional, Proposition~\ref{prop:yfrac} implies that $(y^*,u^*,z^*)\in\operatorname{ext}(\Pi)$. In contrast, the integer separation problem has optimal value $4$. Hence, the linearized separation problem admits a fractional optimal extreme point even when the underlying network problem has a perfect relaxation.
\hfill$\triangle$
\end{example}

Consequently, the separation subproblem for the robust lower-bound formulation remains computationally challenging, even when its deterministic counterpart admits a positive perfect relaxation. The main difficulty lies in repeatedly solving either the bilinear program~\eqref{eq:splb} or the MIP~\eqref{eq:spmip} to identify violated constraints associated with critical structures in $\hat{\Omega}$. To address these difficulties, we develop an approximation framework in the next section.

\section{Approximation for the Robust Lower-Bound Enhancement}
\label{sec:Lower_appx}


We begin by approximating the uncertainty set $\mathcal U$ with a box set that preserves tractability while providing explicit performance guarantees. The following definition formalizes this approximation.

\begin{definition}[Box Set Approximation]
 Given the utility uncertainty set $\mathcal U$, its box set approximation is defined as
 $\hat {\mathcal U}:=\{ u \in\mathbb{R}_+^m \mid \underline u \leq u \leq \bar u\}$, where $\underline u_e:=\min_{u \in \mathcal U}u_e$ and $\bar u_e:=\max_{u \in \mathcal U}u_e$
 denote the entry-wise lower and upper bounds, respectively. The box set approximation of any problem involving utility uncertainty is obtained by replacing $\mathcal U$ with $\hat{\mathcal U}$.
\end{definition}

This approximation scheme leads to a more tractable separation subproblem. In particular, replacing the utility uncertainty set $\mathcal U$ in~\eqref{eq:splb} with its box approximation yields
\begin{equation}
\label{eq:LB-Box}
\begin{aligned}
\min_{y \in \mathcal Y_\Omega}
\langle \underline u + \hat u \cdot x, y \rangle,
\end{aligned}
\end{equation}
which is a deterministic optimization problem.
The approximation improves computational tractability while preserving robustness. Specifically, it eliminates the inner optimization over the uncertainty set and retains the positive perfect relaxation property, when present, while providing a conservative estimate of the worst-case utility associated with each structure.


To address concerns regarding the conservatism introduced by the box set approximation, the rest of this subsection establishes interpretable bounds on the approximation quality.

\begin{lemma}
  \label{lem:appx}
  Suppose $m\ge2$ and $\mathcal U$ is convex, full-dimensional, and satisfies
  $\underline u>0$. Define
  \(
  \alpha_e:=\underline u_e / (\bar u_e-\underline u_e),
  \)
  and let
  $\alpha_{(1)}\leq\alpha_{(2)}\leq\cdots\leq\alpha_{(m)}$
  denote these values in nondecreasing order. Define
  \[
  w_{\mathcal U}(T):=
  \min_{u\in\mathcal U}\iprod{u+\hat u\cdot x,y_T}
  \]
  for any feasible $x$.
  Then, for every $T\in\Omega$,
  \[
  0\leq
  \left(
  1+
  \max_{k=2,\ldots,m}
  \frac{k-1}{\sum_{i\in [k]}\alpha_{(i)}}
  \right)^{-1}
  \leq
  \frac{w_{\hat{\mathcal U}}(T)}{w_{\mathcal U}(T)}
  \leq 1.
  \]
\end{lemma}

\begin{proof}
Since $\mathcal U\subseteq\hat{\mathcal U}\subseteq\mathbb R_+^m$, we have
\(
0\leq
w_{\hat{\mathcal U}}(T)/{w_{\mathcal U}(T)}
\leq 1.
\)
Moreover, under the box approximation $\hat{\mathcal U}$, $\underline u$
is optimal for every fixed feasible $x$ and $y_T$, and hence
\[
w_{\hat{\mathcal U}}(T)
=
\iprod{\underline u+\hat u\cdot x,y_T}.
\]
Since $\hat{\mathcal U}$ is the tightest axis-aligned box containing
$\mathcal U$, for every $e\in E$ there exists a point
$u^e\in\mathcal U$ such that $u^e_e=\underline u_e$.
Because $u^e_j\leq\bar u_j$ for every $j\neq e$,
\[
w_{\mathcal U}(T)
\leq
\iprod{u^e+\hat u\cdot x,y_T}
=
w_{\hat{\mathcal U}}(T)+\iprod{u^e-\underline u,y_T}
\leq
w_{\hat{\mathcal U}}(T)
+\iprod{\bar u-\underline u,y_T}
-(\bar u_e-\underline u_e)y_{T,e}.
\]
Let $b_e:=(\bar u_e-\underline u_e)y_{T,e}\geq0$. Taking the best bound over $e\in E$ gives
\[
w_{\mathcal U}(T)
\leq
w_{\hat{\mathcal U}}(T)
+
\sum_{e \in E} b_e
-
\max_{e\in E} b_e.
\]
Let $B := \max_{e \in E} b_e$. If $B=0$, then the above inequality reduces to $w_{\mathcal U}(T) \leq w_{\hat{\mathcal U}}(T)$. Thus, the desired lower bound follows immediately. Otherwise, let $q_e:= b_e / B$. Then,
\[
\sum_{e\in E} b_e-\max_{e\in E} b_e
= B\left(\sum_{e \in E} q_e - 1\right) =\left(\frac{\sum_{e \in E} q_e - 1}{\sum_{e \in E} \alpha_e q_e}\right)\sum_{e \in E} \alpha_e b_e.
\]
The term in the brackets is a linear-fractional function over the normalized set
$\{q\in[0,1]^E \mid \max_{e\in E} q_e = 1\}$, and its maximum is attained at a nonzero binary vector.
For a binary vector with exactly $k$ nonzero entries,
$$
\left(\frac{\sum_{e \in E} q_e - 1}{\sum_{e \in E} \alpha_e q_e}\right)
=
\frac{k-1}{\sum_{e:q_e=1}\alpha_e}
\leq
\frac{k-1}{\sum_{i \in [k]} \alpha_{(i)}},
$$
where the inequality follows from the ordering of $\alpha_{(i)}$. Therefore,
\[
\sum_{e\in E} b_e-\max_{e\in E} b_e
\leq  \left(\max_{k=2,\dots,m} \frac{k-1}{\sum_{i\in [k]} \alpha_{(i)}}\right)\sum_{e \in E} \alpha_e b_e
= \left(\max_{k=2,\dots,m} \frac{k-1}{\sum_{i\in [k]} \alpha_{(i)}}\right)\iprod{\underline u, y_T}.
\]
Then, we have
$$
\begin{aligned}
  \frac{w_{\hat{\mathcal U}}(T)}{w_{{\mathcal U}}(T)} &\geq \frac{\iprod{\underline u + \hat u \cdot x, y_T}}{\iprod{\underline u + \hat u \cdot x, y_T} + \left(\max_{k=2,\dots,m} \frac{k-1}{\sum_{i\in [k]} \alpha_{(i)}}\right) \iprod{\underline u, y_T}}\\
                                                      &= \frac{1}{1 + \left(\max_{k=2,\dots,m} \frac{k-1}{\sum_{i \in [k]} \alpha_{(i)}}\right) \frac{\iprod{\underline u, y_T}}{\iprod{\underline u + \hat u \cdot x, y_T}}}\\
                                                      &\geq \left(1 + \max_{k=2,\dots, m} \frac{k-1}{\sum_{i \in [k]}\alpha_{(i)}}\right)^{-1},
\end{aligned}
$$
where the last inequality follows from $\iprod{\hat u\cdot x,y_T}\geq0$.
This proves the desired lower bound.
\end{proof}

The quantity $\lambda_{\mathcal U}:=\max_{k=2,\ldots,m}(k-1)/\sum_{i=1}^k \alpha_{(i)}$, where $\alpha_e:=\underline u_e/(\bar u_e-\underline u_e)$, admits an intuitive interpretation in terms of relative uncertainty. Since $(\bar u_e-\underline u_e)/\underline u_e=1/\alpha_e$ measures the relative variation of coordinate $e$, a smaller $\alpha_e$ corresponds to greater uncertainty. The maximization over $k$ identifies the most adverse subset of coordinates whose uncertainties can act jointly. In the homogeneous case, where $\underline u=u_0 1$ and $\bar u=u_1 1$ for some scalars $u_1>u_0>0$, we have $\alpha_e=u_0/(u_1-u_0)$ for all $e$, and hence $\lambda_{\mathcal U}=(m-1) (u_1-u_0)/mu_0$. Thus, the approximation error bound increases with the relative variation in $u$ and, in the homogeneous case, grows with the problem dimension, approaching the uniform bound $(u_1-u_0)/u_0$ as $m\to\infty$.

\begin{proposition}\label{prop:error_order}
Suppose the assumptions of Lemma~\ref{lem:appx} hold.
  The approximation error for the subproblems in \eqref{eq:rlb02} and \eqref{eq:rub02} is of order $O(\lambda_{\mathcal U})$, and satisfies
  \[
  1-\frac{w_{\hat{\mathcal U}}(T)}{w_{\mathcal U}(T)}
  \leq
  \lambda_{\mathcal U}
  \leq
  \frac{m-1}{m}\alpha_{(1)}^{-1}.
  \]
\end{proposition}

\begin{proof}
  By Lemma~\ref{lem:appx},
  \[
  1-\frac{w_{\hat{\mathcal U}}(T)}{w_{\mathcal U}(T)}
  \leq
  \frac{\lambda_{\mathcal U}}{1+\lambda_{\mathcal U}}
  \leq
  \lambda_{\mathcal U},
  \]
  so the approximation error is of order $O(\lambda_{\mathcal U})$. Moreover, since $\alpha_{(1)}\leq\alpha_{(i)}$ for every $i\in[m]$,
  \[
  \lambda_{\mathcal U}
  =
  \max_{k=2,\dots,m}
  \frac{k-1}{\sum_{i\in[k]}\alpha_{(i)}}
  \leq
  \max_{k=2,\dots,m}
  \frac{k-1}{k\alpha_{(1)}}
  =
  \frac{m-1}{m}\alpha_{(1)}^{-1},
  \]
  which concludes the proof.
\end{proof}

The following theorem translates the approximation bound into a performance guarantee for the original optimization problem~\eqref{eq:rlb}.

\begin{theorem}
\label{thm:appx}
Suppose the assumptions of Lemma~\ref{lem:appx} hold.
Let $z_{\mathcal U}(r_{lb})$ denote the optimal value of~\eqref{eq:rlb} under uncertainty set $\mathcal U$ and target value $r_{lb}$. Then, the box approximation $\hat{\mathcal U}$ yields a robust solution to~\eqref{eq:rlb}, and its optimal value satisfies
\(
z_{\mathcal U}(r_{lb})
\leq
z_{\hat{\mathcal U}}(r_{lb})
\leq
z_{\mathcal U}\left(r_{lb}(1+\lambda_{\mathcal U})\right).
\)
\end{theorem}

\begin{proof}
Since $\mathcal U\subseteq\hat{\mathcal U}$, the box approximation underestimates the worst-case utility of every structure $T\in\Omega$. Hence, any solution feasible under $\hat{\mathcal U}$ is also feasible under $\mathcal U$, implying $z_{\mathcal U}(r_{lb})\leq z_{\hat{\mathcal U}}(r_{lb})$.
For the second inequality, Lemma~\ref{lem:appx} gives $w_{\hat{\mathcal U}}(T)\geq (1+\lambda_{\mathcal U})^{-1}w_{\mathcal U}(T)$ for every $T\in\Omega$. Therefore, $w_{\mathcal U}(T)\geq r_{lb}(1+\lambda_{\mathcal U})$ implies $w_{\hat{\mathcal U}}(T)\geq r_{lb}$. Thus, every solution feasible under $\mathcal U$ with target $r_{lb}(1+\lambda_{\mathcal U})$ is feasible under $\hat{\mathcal U}$ with target $r_{lb}$. Since~\eqref{eq:rlb} is a minimization problem, $z_{\hat{\mathcal U}}(r_{lb})\leq z_{\mathcal U}(r_{lb}(1+\lambda_{\mathcal U}))$, which concludes the proof.
\end{proof}

According to Theorem~\ref{thm:appx}, the box approximation is more conservative than the original formulation under $\mathcal U$, but no more conservative than solving~\eqref{eq:rlb} under $\mathcal U$ with the inflated target $r_{\mathrm{lb}}(1+\lambda_{\mathcal U})$. Thus, $\lambda_{\mathcal U}$ directly quantifies the price of tractability introduced by the box approximation.

\begin{example}
If $\mathcal U$ is an $m$-dimensional Euclidean ball of radius $0.1$ centered at $1$, then $\alpha_e=4.5$ for all $e\in E$, so $\lambda_{\mathcal U}=(m-1)/(4.5m)$. Hence,
$\frac{w_{\hat{\mathcal U}}(T)}{w_{\mathcal U}(T)}
\geq
\frac{4.5m}{5.5m-1}.$
By Theorem~\ref{thm:appx}, the box approximation is therefore no more conservative than increasing the target value by a factor of $1+\lambda_{\mathcal U}$, which is close to $1.111$ when $m=2$ and $1.222$ as $m\to\infty$. Equivalently, the approximation ratio decreases from $90\%$ to a limiting guarantee of $81.82\%$. Thus, even as the problem dimension grows, the conservatism introduced by the box approximation remains uniformly bounded.
\hfill$\triangle$
\end{example}

\section{Robust Upper-Bound Enhancement} \label{section:robut_upper_bound}
This section focuses on the robust upper-bound problem~\eqref{eq:rub}. We develop two reformulations, compare the tightness of their feasible regions, and introduce a box-set approximation reformulation.

\subsection{Linearized Reformulation}
By applying standard reformulation techniques from robust optimization
and mixed integer programming, problem \eqref{eq:rub} can be expressed
as the following formulation.

\begin{theorem}
 Problem \eqref{eq:rub} is equivalent to the following formulation:
\setlength{\abovedisplayskip}{2pt}
    \setlength{\abovedisplayshortskip}{2pt}
    \setlength{\belowdisplayskip}{2pt}
    \setlength{\belowdisplayshortskip}{2pt}
 \begin{subequations}
    \label{eq:ubmip}
  \begin{align}
    \text{Lin}: \min_{\eta,\tau,z \geq 0, x \in \mathcal X, y\in \mathcal Y_\Omega} &~ \langle g, \eta \rangle \label{eq:ubmip00}\\
    \text{s.t.} &~ G^\intercal \eta \geq x\label{eq:ubmip01}\\
                &~-H^\intercal\tau \leq y \label{eq:ubmip02} \\
                &~-\langle h, \tau \rangle + \langle z, 1\rangle \geq r_{ub} \label{eq:ubmip03} \\
                &~z \leq \hat u \cdot x \label{eq:ubmip04} \\
                &~z \leq \hat u \cdot y. \label{eq:ubmip05} 
  \end{align}
  \end{subequations}
\end{theorem}
\begin{proof}
    Constraints \eqref{eq:ubmip00} and \eqref{eq:ubmip01} are obtained by
    fixing $x$ and dualizing the inner optimization problem over $c$, with
    $\eta$ denoting the corresponding dual variables. Similarly, after
    fixing $x$ and $y$ at the outer level, we dualize the inner problem in
    \eqref{eq:rub02}, with $\tau$ denoting the associated dual variables.
    This yields constraints \eqref{eq:ubmip02} and \eqref{eq:ubmip03}.
    Combining these constraints gives the following formulation.
    $$
      \begin{aligned}
        \min_{\eta \geq 0,\tau \geq 0, x \in \mathcal X, y\in \mathcal Y_\Omega} &~ \langle g, \eta \rangle\label{eq:rubreform00}\\
        \text{s.t.} &~ G^\intercal \eta \geq x\label{eq:rubreform01}\\
                    &~-H^\intercal\tau \leq y \label{eq:rubreform02} \\
                    &~-\langle h, \tau \rangle + \langle \hat u \cdot x, y \rangle  \geq r_{ub}. \label{eq:rubreform03}
      \end{aligned}
      $$
Since the last constraint contains a bilinear term, we apply
a standard mixed integer linearization to obtain constraints
\eqref{eq:ubmip03}--\eqref{eq:ubmip05}.
\end{proof}

\subsection{Inclusion Reformulation}

Intuitively, the upper-bound enhancement problem \eqref{eq:rub} has no incentive to
enhance elements that are not selected in the maximum utility structure
of the inner problem~\eqref{eq:rub02}. This observation motivates the following
\emph{inclusion reformulation}.
 \setlength{\abovedisplayskip}{2pt}
\setlength{\abovedisplayshortskip}{2pt}
\setlength{\belowdisplayskip}{2pt}
\setlength{\belowdisplayshortskip}{2pt}
\begin{subequations} \label{eq:ssubmip}
\begin{align}
    \text{Incl}: \min_{\eta,\tau \geq 0, x \in \mathcal X, y\in \mathcal Y_\Omega} \ &~ \langle g, \eta \rangle \label{eq:ssubmip00} \\
    \text{s.t.} &~ G^\intercal \eta \geq x\label{eq:ssubmip01}\\
                &~-H^\intercal\tau \leq y \label{eq:ssubmip02} \\
                &~-\langle h, \tau \rangle + \langle \hat u, x \rangle \geq r_{ub} \label{eq:ssubmip03} \\
                &~x \leq y \label{eq:ssubmip04},
\end{align}
\end{subequations}
where constraint \eqref{eq:ssubmip04} ensures that the selected structure includes every enhanced element. 
The following proposition establishes the correctness of~\eqref{eq:ssubmip} under the assumption that $\mathcal X$ is downward closed, i.e., $x\in\mathcal X$ and $x'\leq x$ imply $x'\in\mathcal X$ for any binary vector $x'$. This property holds, for example, when $\mathcal X=\{0,1\}^m$ or is defined by a budget constraint with nonnegative costs.


\begin{proposition} \label{prop:optlin}
  Suppose $\mathcal X$ is downward closed and \eqref{eq:ubmip} is feasible. Then there exists an optimal solution $(x^*, y^*)$ of \eqref{eq:ubmip} such that $x^* \leq y^*$.
\end{proposition}

\begin{proof}
Let $(z^*, x^*,y^*,\eta^*,\tau^*)$ be an optimal solution to
\eqref{eq:ubmip}. If $x^*\leq y^*$, the result follows immediately.
Suppose instead that $x^*\nleq y^*$. Then there exists at least one
edge $e\in E$ such that $x_e^*>y_e^*$. Since $x^*$ and $y^*$ are
binary, this implies that $x_e^*=1$ and $y_e^*=0$.
Let $E'=\{e\in E \mid x_e^*=1,\ y_e^*=0\}$, and define
$\tilde{x}_e=0$ for $e\in E'$ and $\tilde{x}_e=x_e^*$ otherwise. By
construction, $\tilde{x}\leq y^*$. 
Since $\mathcal X$ is downward closed, $\tilde x\in\mathcal X$. 
Moreover, removing enhancements from edges outside $y^\ast$ preserves the robust utility constraint. Since $\tilde x\le x^\ast$ and $c\ge0$ for every $c\in\mathcal C$, we have
\[
\max_{c\in\mathcal C}\langle c,\tilde x\rangle
\le
\max_{c\in\mathcal C}\langle c,x^\ast\rangle.
\]
Thus, the modified enhancement decision remains feasible and does not increase the objective value. By optimality of $x^\ast$, there exists an optimal solution satisfying $\tilde x\le y^*$.
\end{proof}

Next, we compare the strength of the linear relaxations of~\eqref{eq:ssubmip} and~\eqref{eq:ubmip}.
Let $\bar{\mathcal X}$ denote a linear relaxation of $\mathcal X$. We define the following two polyhedra
\[
 \mathcal{P}_{Lin} := \{(z,x,y,\eta,\tau) \in \mathbb{R}_+^m \times \bar{\mathcal X} \times  \Bar{\mathcal Y}_\Omega \times \mathbb{R}_+^k \times \mathbb{R}_+^\ell, \text{satisfying \eqref{eq:ubmip01}--\eqref{eq:ubmip05}}\},
 \]
 \[
  \mathcal{P}_{Incl} := \{(x,y,\eta,\tau) \in \bar{\mathcal X} \times \Bar{\mathcal Y}_\Omega \times \mathbb{R}_+^k \times \mathbb{R}_+^\ell, \text{satisfying \eqref{eq:ssubmip01}--\eqref{eq:ssubmip04}}\}.
\]
The following theorem establishes the relationship between these two feasible regions, where $\operatorname{proj}(\cdot)$ denotes the projection operator.

\begin{theorem} \label{thm:strsub}
We have
$
\mathcal{P}_{Incl} \subseteq \operatorname{proj}_{(x,y,\eta,\tau)}(\mathcal{P}_{Lin}).
$
Moreover, the inclusion is strict if the robust upper-bound problem~\eqref{eq:rub} admits a feasible solution pair $(x,y)$ with $x \not\le y$.
\end{theorem}
\begin{proof}
Let $(x,y,\eta,\tau)$ be feasible to~\eqref{eq:ssubmip}, and define $z:=\hat u\cdot x$. Then constraints~\eqref{eq:ubmip01}--\eqref{eq:ubmip04} follow directly from~\eqref{eq:ssubmip01}--\eqref{eq:ssubmip03}. Moreover, \eqref{eq:ssubmip04} gives $x\le y$, and hence $\hat u\cdot x\le \hat u\cdot y$. Therefore, $z$ also satisfies~\eqref{eq:ubmip05}, so $(z,x,y,\eta,\tau)$ is feasible to~\eqref{eq:ubmip}. This proves $\mathcal P_{Incl}\subseteq \operatorname{proj}_{(x,y,\eta,\tau)}(\mathcal P_{Lin})$. The strictness condition then follows directly from the definitions of the two polyhedra.
\end{proof}

Theorem~\ref{thm:strsub} establishes that the inclusion formulation has a linear relaxation at least as strong as that of the linearized formulation, with strict dominance under the stated condition. The inclusion formulation also requires fewer variables and constraints, which offer computational advantages.



\subsection{Approximation for the Robust Upper-Bound Enhancement}\label{subsec:Upper_appx}
The approximation approach developed in Section~\ref{sec:Lower_appx} also applies to the robust upper-bound problem. In particular, replacing the uncertainty set $\mathcal U$ in~\eqref{eq:ssubmip} with its box approximation $\hat{\mathcal U}$ eliminates the optimization over $\mathcal U$, yielding the following simplified approximation formulation.
\setlength{\abovedisplayskip}{2pt}
\setlength{\abovedisplayshortskip}{2pt}
\setlength{\belowdisplayskip}{2pt}
\setlength{\belowdisplayshortskip}{2pt}
\begin{subequations} \label{eq:UB-Inc-Box}
\begin{align}
    \min_{\eta \geq 0,\, x \in \mathcal X,\, y \in \mathcal Y_{\Omega}}
    &~ \langle g, \eta \rangle
    \label{eq:UB-Inc-Box00} \\
    \text{s.t.}\quad
    &~ G^\intercal \eta \geq x,
    \label{eq:UB-Inc-Box01}\\
    &~ x \leq y,
    \label{eq:UB-Inc-Box04}\\
    &~ \langle \underline u + \hat u \cdot x, y\rangle
    \geq r_{ub}.
    \label{eq:UB-Inc-Box03}
\end{align}
\end{subequations}
This substitution eliminates the variables $\tau$ and the associated constraint set~\eqref{eq:ssubmip02}, resulting in a more compact formulation that can potentially improve computational efficiency.

Since Lemma~\ref{lem:appx} applies independently of whether the lower-bound or upper-bound problem is considered, the box approximation of the upper-bound problem also admits a similar performance guarantee, as established below.

\begin{corollary}
\label{cor:upper_appx}
Suppose the assumptions of Lemma~\ref{lem:appx}
hold and $\mathcal X$ is downward closed.
Let $z_{\mathcal U}(r_{ub})$ denote the optimal value of~\eqref{eq:rub}
under uncertainty set $\mathcal U$ and target value $r_{ub}$. Then,
the box set approximation $\hat{\mathcal U}$ yields a robust solution to~\eqref{eq:rub}, and its optimal value satisfies
$
z_{\mathcal U}(r_{ub})
\leq
z_{\hat{\mathcal U}}(r_{ub})
\leq
z_{\mathcal U}\left(
r_{ub}\left(1+\lambda_{\mathcal U}\right)
\right).
$
\end{corollary}

\begin{proof}
Since $\hat{\mathcal U}$ provides a utility underestimation for every
structure $T\in\Omega$, replacing $\mathcal U$ with
$\hat{\mathcal U}$ reduces the feasible solution space of
\eqref{eq:rub}. Therefore,
$
z_{\mathcal U}(r_{ub})
\leq
z_{\hat{\mathcal U}}(r_{ub}),
$
and every solution feasible under $\hat{\mathcal U}$ is also feasible
under $\mathcal U$.
For the second inequality, by Lemma~\ref{lem:appx},
$
w_{\hat{\mathcal U}}(T)
\geq
\left(1+\lambda_{\mathcal U}\right)^{-1}
w_{\mathcal U}(T)
$
for every $T\in\Omega$. Hence, requiring
$w_{\mathcal U}(T)\geq
r_{ub}(1+\lambda_{\mathcal U})$
ensures that
$w_{\hat{\mathcal U}}(T)\geq r_{ub}$.
Consequently, the feasible solution space associated with the target
$r_{ub}(1+\lambda_{\mathcal U})$ under $\mathcal U$ is contained in the
feasible solution space associated with the target $r_{ub}$ under
$\hat{\mathcal U}$, which yields
$
z_{\hat{\mathcal U}}(r_{ub})
\leq
z_{\mathcal U}\left(
r_{ub}(1+\lambda_{\mathcal U})
\right).
$
\end{proof}

According to Corollary~\ref{cor:upper_appx}, the box set approximation yields a formulation that is more conservative than the original robust upper-bound problem under $\mathcal U$, while remaining less conservative than increasing the target value to $r_{ub}(1+\lambda_{\mathcal U})$ in~\eqref{eq:rub}. 
Moreover, Proposition~\ref{prop:error_order} gives $\lambda_{\mathcal U}\leq (m-1)/(m\alpha_{(1)})$, so as $m\to\infty$, the error factor $\lambda_{\mathcal U}$ is uniformly bounded by $\alpha_{(1)}^{-1}=\max_{e\in E}(\bar u_e-\underline u_e)/\underline u_e$.


\section{Numerical Experiments}
\label{Section Numerical Exper}
This section conducts computational experiments to evaluate the efficiency and solution quality of the proposed methods for the lower- and upper-bound variants of the robust network enhancement problem. The reported performance measures include the average solution time, optimality gap, and ratio of the approximation objective value to the corresponding exact value, together with $\lambda_{\mathcal U}$, which is used to assess the approximation bound through Theorem~\ref{thm:appx} and Corollary~\ref{cor:upper_appx}.

We generate three independent network instances for each configuration with
$n$ vertices and $m \in \{2n,3n\}$ edges. To model the spanning tree set
$\mathcal{Y}_{\Omega}$, we adopt a single-commodity flow formulation based on
the network-flow framework for spanning trees described by
\citet{magnanti1995optimal}. Let $v_0 \in V$
denote an arbitrarily selected root vertex. For each undirected edge
$e=\{i,j\}\in E$, let $y_e$ indicate whether edge $e$ is selected, and let
$f_{ij}$ and $f_{ji}$ denote the flows along the two corresponding directed
arcs. The spanning tree set $\mathcal{Y}_{\Omega}$ is then represented by
\begin{subequations}
\label{eq:single_commodity_mst}
\begin{align}
&~ \sum_{\{j:\{i,j\}\in E\}} f_{ij}
-
\sum_{\{j:\{i,j\}\in E\}} f_{ji}
= b_i,
\quad \forall i\in V,
\label{eq:single_commodity_mst_flow}\\
&~ f_{ij}+f_{ji}
\leq (n-1)y_e,
\quad \forall e=\{i,j\}\in E,
\label{eq:single_commodity_mst_link}\\
&~ \sum_{e\in E} y_e
= n-1,
\label{eq:single_commodity_mst_card}\\
&~ f_{ij},f_{ji}
\geq 0,
\quad \forall e=\{i,j\}\in E,
\label{eq:single_commodity_mst_nonneg}\\
&~ y_e
\in\{0,1\},
\quad \forall e\in E,
\label{eq:single_commodity_mst_binary}
\end{align}
\end{subequations}
where $b_i=n-1$ if $i=v_0$, and $b_i=-1$ otherwise.
The root vertex $v_0$ sends $n-1$ units of a single
commodity, with each remaining vertex receiving one unit.
Constraints~\eqref{eq:single_commodity_mst_flow}
enforce flow balance and, together with the linking
constraints~\eqref{eq:single_commodity_mst_link},
ensure that every vertex is reachable from the root.
Constraint~\eqref{eq:single_commodity_mst_card}
requires exactly $n-1$ edges to be selected. 
Constraints~\eqref{eq:single_commodity_mst_nonneg}
and~\eqref{eq:single_commodity_mst_binary}
impose nonnegativity and binary restrictions,
respectively.

Uncertainty in the utility and cost parameters is modeled using standard
budgeted uncertainty sets commonly employed in the robust optimization
literature \citep{bertsimas2004price,busing2025recycling}. 
Specifically, the utility uncertainty set is given by
\[
\mathcal{U}
=
\left\{
u \in \mathbb{R}_+^m \;\Bigm|\;
(1-\rho)\tilde{u} \leq u \leq \tilde{u},
\quad
\sum_{e \in E}
\frac{\tilde{u}_e-u_e}{\rho\tilde{u}_e}
\leq \Gamma
\right\},
\]
whereas the cost uncertainty set is given by
\[
\mathcal{C}
=
\left\{
c \in \mathbb{R}_+^m \;\Bigm|\;
\tilde{c} \leq c \leq (1+\rho)\tilde{c},
\quad
\sum_{e \in E}
\frac{c_e-\tilde{c}_e}{\rho\tilde{c}_e}
\leq \Gamma
\right\}.
\]
The nominal utility $\tilde{u}$ and cost $\tilde{c}$ are generated from $[20,80]^m$. The enhanced utility vector is
defined as $\hat{u}=0.2\tilde{u}$.
The parameter $\rho\in\{0.01,0.05\}$ controls the maximum relative deviation
of each uncertain parameter from its nominal value. The parameter $\Gamma\in[0,m]$ specifies the total uncertainty budget. In all experiments,
we set $\Gamma=0.2m$ as the uncertainty budget level.
For the lower- and upper-bound enhancement problems, respectively, the enhancement target $r$ is set to $(1+\gamma)$ times the nominal minimum or maximum spanning-tree value, where $\gamma \in \{0.05,0.10\}$.

All experiments were conducted on a 2021 MacBook Pro equipped with an Apple
M1 Max processor and 64\,GB of memory, running macOS Sonoma
version~14.6.1. The proposed methods were implemented in Python~3.13 and solved using Gurobi~13.0, with a time limit of 3,600 seconds.

\subsection{Results for the Robust Lower-Bound Enhancement}
\label{lower bound results}

For the
robust lower-bound problem, we consider
$n\in\{20,40,60,80,100\}$ vertices. We solve the master problem
\eqref{eq:rlbreform} with two alternative separation procedures:
(i) the linearized MIP separation subproblem \eqref{eq:spmip}, referred to as
LB-MIP; and
(ii) the separation subproblem \eqref{eq:LB-Box}, in which the utility uncertainty
set $\mathcal{U}$ is replaced by its box approximation, referred to as
LB-Box. Table~\ref{tab:computational_results} provides a detailed comparison between the exact LB-MIP method and the approximate LB-Box method.

In terms of efficiency, 
LB-Box generally requires less time than LB-MIP on instances solved to optimality.
For example, for $(n,m)=(60,180)$, $\gamma=0.10$, and $\rho=0.01$, the average
solution time decreases from $257.46$ seconds for LB-MIP to
$4.53$ seconds for LB-Box. Similarly, for $(n,m)=(80,240)$,
$\gamma=0.05$, and $\rho=0.01$, LB-Box requires only $16.94$ seconds,
compared with $1759.28$ seconds for LB-MIP.
For larger and more challenging configurations where both methods reach the time limit, LB-Box generally obtains smaller gaps. For
instance, for $(n,m)=(100,300)$, the average gaps of LB-MIP range from
$4.04\%$ to $16.34\%$, compared with $0.00\%$ to $6.18\%$ for
LB-Box. Overall, the approximate
LB-Box method is more computationally efficient than
LB-MIP, which is expected as LB-Box adopts efficient minimum spanning tree algorithms for separating constraints through approximation.

The Opt. ratio column reports the average ratio of the objective value obtained by LB-Box to that obtained by LB-MIP, providing an empirical measure of approximation quality. As shown in Table~\ref{tab:computational_results}, this ratio generally remains around $1.1$ when $\rho=0.01$, while larger deviations, such as $1.51$ for $(n,m)=(80,160)$, occur when $\rho=0.05$. On the other hand, $\lambda_{\mathcal U}$ theoretically quantifies the approximation quality, as Proposition~\ref{prop:error_order} establishes that the approximation error for the subproblems in \eqref{eq:rlb02} and \eqref{eq:rub02} is of order $O(\lambda_{\mathcal U})$. Consistent with Theorem~\ref{thm:appx}, smaller values of $\lambda_{\mathcal U}$ yield tighter approximation guarantees, whereas larger values permit potentially greater approximation errors. 

Therefore, LB-Box provides a balance between computational efficiency and solution quality. It reduces solution times and final optimality gaps while maintaining objective values close to those of LB-MIP.

\begin{table}[!tbp]
\TABLE
{\raggedright
Comparison of the box set approximation and linearized MIP formulations for the
robust lower-bound problem.
\label{tab:computational_results}
}
{
\scriptsize
\setlength{\tabcolsep}{9pt}
\renewcommand{\arraystretch}{1}

\makebox[\textwidth][c]{%
\begin{tabular}{ccccccccc}
\toprule

\multicolumn{3}{c}{Configuration}
&
\multicolumn{2}{c}{Time (s)}
&
\multicolumn{2}{c}{Gap (\%)}
&
\multirow{2}{*}{Opt. ratio}
&
\multirow{2}{*}{$\lambda_{\mathcal U}$}
\\

\cmidrule(lr){1-3}
\cmidrule(lr){4-5}
\cmidrule(lr){6-7}

$(n,m)$
&
$\gamma$
&
$\rho$
&
LB-MIP
&
LB-Box
&
LB-MIP
&
LB-Box
&
&
\\

\midrule


\multirow{4}{*}{$(20,40)$}
&
\multirow{2}{*}{$0.05$}
&
$0.01$
&
$0.33$ & \textbf{0.04} & \textbf{0.00} & \textbf{0.00} & $1.13$ & $0.01$
\\

&
&
$0.05$
&
$0.53$ & \textbf{0.06} & \textbf{0.00} & \textbf{0.00} & $1.39$ & $0.05$
\\

\cmidrule(lr){2-9}

&
\multirow{2}{*}{$0.10$}
&
$0.01$
&
$0.64$ & \textbf{0.04} & \textbf{0.00} & \textbf{0.00} & $1.04$ & $0.01$
\\

&
&
$0.05$
&
$0.82$ & \textbf{0.06} & \textbf{0.00} & \textbf{0.00} & $1.28$ & $0.05$
\\

\cmidrule(lr){1-9}


\multirow{4}{*}{$(20,60)$}
&
\multirow{2}{*}{$0.05$}
&
$0.01$
&
$0.63$ & \textbf{0.06} & \textbf{0.00} & \textbf{0.00} & $1.07$ & $0.01$
\\

&
&
$0.05$
&
$1.00$ & \textbf{0.09} & \textbf{0.00} & \textbf{0.00} & $1.24$ & $0.05$
\\

\cmidrule(lr){2-9}

&
\multirow{2}{*}{$0.10$}
&
$0.01$
&
$1.59$ & \textbf{0.09} & \textbf{0.00} & \textbf{0.00} & $1.03$ & $0.01$
\\

&
&
$0.05$
&
$2.37$ & \textbf{0.13} & \textbf{0.00} & \textbf{0.00} & $1.14$ & $0.05$
\\

\cmidrule(lr){1-9}


\multirow{4}{*}{$(40,80)$}
&
\multirow{2}{*}{$0.05$}
&
$0.01$
&
$1.65$ & \textbf{0.07} & \textbf{0.00} & \textbf{0.00} & $1.11$ & $0.01$
\\

&
&
$0.05$
&
$3.40$ & \textbf{0.33} & \textbf{0.00} & \textbf{0.00} & $1.46$ & $0.05$
\\

\cmidrule(lr){2-9}

&
\multirow{2}{*}{$0.10$}
&
$0.01$
&
$10.21$ & \textbf{0.35} & \textbf{0.00} & \textbf{0.00} & $1.06$ & $0.01$
\\

&
&
$0.05$
&
$14.67$ & \textbf{0.67} & \textbf{0.00} & \textbf{0.00} & $1.30$ & $0.05$
\\

\cmidrule(lr){1-9}


\multirow{4}{*}{$(40,120)$}
&
\multirow{2}{*}{$0.05$}
&
$0.01$
&
$6.28$ & \textbf{0.14} & \textbf{0.00} & \textbf{0.00} & $1.07$ & $0.01$
\\

&
&
$0.05$
&
$20.92$ & \textbf{1.58} & \textbf{0.00} & \textbf{0.00} & $1.26$ & $0.05$
\\

\cmidrule(lr){2-9}

&
\multirow{2}{*}{$0.10$}
&
$0.01$
&
$39.75$ & \textbf{2.20} & \textbf{0.00} & \textbf{0.00} & $1.04$ & $0.01$
\\

&
&
$0.05$
&
$66.35$ & \textbf{14.34} & \textbf{0.00} & \textbf{0.00} & $1.16$ & $0.05$
\\

\cmidrule(lr){1-9}


\multirow{4}{*}{$(60,120)$}
&
\multirow{2}{*}{$0.05$}
&
$0.01$
&
$30.87$ & \textbf{0.58} & \textbf{0.00} & \textbf{0.00} & $1.11$ & $0.01$
\\

&
&
$0.05$
&
$105.71$ & \textbf{12.93} & \textbf{0.00} & \textbf{0.00} & $1.44$ & $0.05$
\\

\cmidrule(lr){2-9}

&
\multirow{2}{*}{$0.10$}
&
$0.01$
&
$313.56$ & \textbf{23.23} & \textbf{0.00} & \textbf{0.00} & $1.07$ & $0.01$
\\

&
&
$0.05$
&
$503.77$ & \textbf{264.19} & \textbf{0.00} & \textbf{0.00} & $1.31$ & $0.05$
\\

\cmidrule(lr){1-9}


\multirow{4}{*}{$(60,180)$}
&
\multirow{2}{*}{$0.05$}
&
$0.01$
&
$37.65$ & \textbf{0.57} & \textbf{0.00} & \textbf{0.00} & $1.07$ & $0.01$
\\

&
&
$0.05$
&
$130.83$ & \textbf{2.88} & \textbf{0.00} & \textbf{0.00} & $1.25$ & $0.05$
\\

\cmidrule(lr){2-9}

&
\multirow{2}{*}{$0.10$}
&
$0.01$
&
$257.46$ & \textbf{4.53} & \textbf{0.00} & \textbf{0.00} & $1.04$ & $0.01$
\\

&
&
$0.05$
&
$1341.41$ & \textbf{175.21} & $0.30$ & \textbf{0.00} & $1.17$ & $0.05$
\\

\cmidrule(lr){1-9}


\multirow{4}{*}{$(80,160)$}
&
\multirow{2}{*}{$0.05$}
&
$0.01$
&
$1231.58$ & \textbf{0.46} & $26.72$ & \textbf{0.00} & $1.04$ & $0.01$
\\

&
&
$0.05$
&
$1432.25$ & \textbf{147.94} & \textbf{0.00}$^{*}$ & \textbf{0.00} & $1.51^* $ & $0.05$
\\

\cmidrule(lr){2-9}

&
\multirow{2}{*}{$0.10$}
&
$0.01$
&
$2416.11$ & \textbf{278.58} & $1.95^*$ & \textbf{0.00} & $1.06^*$ & $0.01$
\\

&
&
$0.05$
&
$2444.24$ & \textbf{1505.17} & $4.79^*$ & \textbf{1.52} & $1.30^* $ & $0.05$
\\

\cmidrule(lr){1-9}


\multirow{4}{*}{$(80,240)$}
&
\multirow{2}{*}{$0.05$}
&
$0.01$
&
$1759.28$ & \textbf{16.94} & \textbf{0.00} & \textbf{0.00} & $1.07$ & $0.01$
\\

&
&
$0.05$
&
$3418.73$ & \textbf{1157.81} & $8.24$ & \textbf{0.00} & $1.19$ & $0.05$
\\

\cmidrule(lr){2-9}

&
\multirow{2}{*}{$0.10$}
&
$0.01$
&
$3087.95$ & \textbf{1047.47} & $10.14$ & \textbf{0.00} & $0.99$ & $0.01$
\\

&
&
$0.05$
&
$3600.00$ & \textbf{3540.79} & $12.53$ & \textbf{1.52} & $1.10$ & $0.05$
\\

\cmidrule(lr){1-9}


\multirow{4}{*}{$(100,200)$}
&
\multirow{2}{*}{$0.05$}
&
$0.01$
&
$72.95$ & \textbf{1.49} & \textbf{0.00} & \textbf{0.00} & $1.12$ & $0.01$
\\

&
&
$0.05$
&
\textbf{349.50} & $376.67$ & \textbf{0.00} & \textbf{0.00} & $1.47$ & $0.05$
\\

\cmidrule(lr){2-9}

&
\multirow{2}{*}{$0.10$}
&
$0.01$
&
$2348.66$ & \textbf{887.80} & $1.64$ & \textbf{0.00} & $1.07$ & $0.01$
\\

&
&
$0.05$
&
$3600.00$ & \textbf{1610.91} & $5.90$ & \textbf{1.07} & $1.29$ & $0.05$
\\

\cmidrule(lr){1-9}


\multirow{4}{*}{$(100,300)$}
&
\multirow{2}{*}{$0.05$}
&
$0.01$
&
$2479.44$ & \textbf{217.42} & $4.04$ & \textbf{0.00} & $1.06$ & $0.01$
\\

&
&
$0.05$
&
\textbf{3600.00} & \textbf{3600.00} & $8.60^*$ & \textbf{3.43} & $1.23^*$ & $0.05$
\\

\cmidrule(lr){2-9}

&
\multirow{2}{*}{$0.10$}
&
$0.01$
&
\textbf{3600.00} & \textbf{3600.00} & $10.05^*$ & \textbf{2.98} & $1.00^*$ & $0.01$
\\

&
&
$0.05$
&
\textbf{3600.00} & \textbf{3600.00} & $16.34^*$ & \textbf{6.18} & $1.08^*$ & $0.05$
\\

\bottomrule
\end{tabular}%
}
}
{}\begingroup\fontsize{8}{12}\selectfont\emph{Notes.}
The ``Opt. ratio'' column reports the average ratio of the objective value obtained by \emph{LB-Box} to the corresponding incumbent objective value obtained by \emph{LB-MIP}. The $\lambda_{\mathcal U}$ column measures the potential approximation error induced by the box approximation. An asterisk $\ast$ indicates that one of the three instances had no feasible solution within the 3600-second time limit; the reported average is based on the remaining two instances.
\par
\endgroup
\end{table}

\subsection{Results for the Robust Upper-Bound Enhancement}

\label{subsec:upper-bound-results}

The computational study of the robust upper-bound problem considers networks
with $n \in \{100,\allowbreak 200,\allowbreak 300,\allowbreak
400,\allowbreak 500\}$. Three solution approaches are evaluated:
\textnormal{(i)} the linearized reformulation \eqref{eq:ubmip}, referred to as
UB-Lin; \textnormal{(ii)} the inclusion-based reformulation
\eqref{eq:ssubmip}, referred to as UB-Incl; and
\textnormal{(iii)} the inclusion-based reformulation \eqref{eq:UB-Inc-Box} with box approximation, referred to as UB-Incl-Box. 

Table~\ref{tab:upper_bound_computational_results} shows that the
performance of the three formulations varies across the different configurations.
On smaller instances, all three formulations are solved to
optimality with comparable running times, with UB-Incl-Box generally being the most efficient. For example, for $(n,m)=(200,400)$, $\gamma=0.05$, and $\rho=0.01$,
UB-Incl-Box requires $21.12$ seconds, compared with $56.05$ seconds for
UB-Lin and $456.11$ seconds for UB-Incl. Similarly, for $(n,m)=(200,600)$, $\gamma=0.10$, and $\rho=0.05$, UB-Incl-Box solves the problem in $20.94$ seconds, versus $26.54$ seconds for UB-Lin and $29.03$ seconds for UB-Incl.
For larger configurations, several formulations reach the time limit. In these cases, the gaps
remain small across all three formulations and are all below $1\%$.
The largest reported gap for UB-Incl-Box is $0.56\%$, compared with
$0.76\%$ for UB-Lin and $0.67\%$ for UB-Incl. Moreover,
UB-Incl-Box yields the smallest gap in most large
configurations. For example, for $(n,m)=(500,1000)$, $\gamma=0.10$, and
$\rho=0.05$, its average gap is $0.37\%$, compared with $0.44\%$ and
$0.46\%$ for UB-Lin and UB-Incl, respectively. 


The Opt. ratio column exhibits a pattern of empirical approximation quality similar to that observed for the lower-bound problem. The ratio generally remains around $1.1$ when $\rho=0.01$, whereas larger deviations occur when $\rho=0.05$, reaching $1.53$ for $(n,m)=(500,1000)$. The $\lambda_{\mathcal U}$ column provides the corresponding theoretical measure. Consistent with Corollary~\ref{cor:upper_appx}, smaller values of $\lambda_{\mathcal U}$ yield tighter approximation guarantees, whereas larger values allow greater approximation errors. 
However, for comparable values of $\lambda_{\mathcal U}$, the objective ratios in the robust upper-bound experiments are often slightly larger than those in the robust lower-bound experiments, with average ratios of 1.21 and 1.17, respectively.
This indicates that, although $\lambda_{\mathcal U}$ captures the same potential error induced by the box approximation, the realized approximation quality also depends on the sensitivity of the underlying enhancement problem. In this respect, the robust upper-bound problem appears to be more sensitive to the box approximation than the robust lower-bound problem, at least for the instances tested.

Overall, the three formulations exhibit comparable computational performance for the robust upper-bound problem, while UB-Incl-Box often achieves shorter solution times at the cost of a bounded loss in solution quality.

\begin{table}[!tbp]
\TABLE
{\raggedright
Computational comparison of the proposed formulations for the robust
upper-bound problem.
\label{tab:upper_bound_computational_results}
}
{
\scriptsize
\setlength{\tabcolsep}{6pt}
\renewcommand{\arraystretch}{1.2}

\makebox[\textwidth][c]{%
\begin{tabular}{ccccccccccc}
\toprule

\multicolumn{3}{c}{Configuration}
&
\multicolumn{3}{c}{Time (s)}
&
\multicolumn{3}{c}{Gap (\%)}
&
\multirow{2}{*}{Opt. ratio}
&
\multirow{2}{*}{$\lambda_{\mathcal U}$}
\\

\cmidrule(lr){1-3}
\cmidrule(lr){4-6}
\cmidrule(lr){7-9}

$(n,m)$
&
$\gamma$
&
$\rho$
&
UB-Lin
&
UB-Incl
&
UB-Incl-Box
&
UB-Lin
&
UB-Incl
&
UB-Incl-Box
&
&
\\

\midrule


\multirow{4}{*}{$(100,200)$}
&
\multirow{2}{*}{$0.05$}
&
$0.01$
&
$1.72$ & $1.73$ & \textbf{1.11} & \textbf{0.00} & \textbf{0.00} & \textbf{0.00} & $1.12$ & $0.01$
\\

&
&
$0.05$
&
$1.55$ & $1.50$ & \textbf{1.27} & \textbf{0.00} & \textbf{0.00} & \textbf{0.00} & $1.52$ & $0.05$
\\

\cmidrule(lr){2-11}

&
\multirow{2}{*}{$0.10$}
&
$0.01$
&
$1.45$ & $1.44$ & \textbf{1.09} & \textbf{0.00} & \textbf{0.00} & \textbf{0.00} & $1.08$ & $0.01$
\\

&
&
$0.05$
&
$2.15$ & $1.53$ & \textbf{1.31} & \textbf{0.00} & \textbf{0.00} & \textbf{0.00} & $1.34$ & $0.05$
\\

\cmidrule(lr){1-11}


\multirow{4}{*}{$(100,300)$}
&
\multirow{2}{*}{$0.05$}
&
$0.01$
&
$4.26$ & $4.73$ & \textbf{2.98} & \textbf{0.00} & \textbf{0.00} & \textbf{0.00} & $1.08$ & $0.01$
\\

&
&
$0.05$
&
$3.99$ & \textbf{2.86} & $3.83$ & \textbf{0.00} & \textbf{0.00} & \textbf{0.00} & $1.30$ & $0.05$
\\

\cmidrule(lr){2-11}

&
\multirow{2}{*}{$0.10$}
&
$0.01$
&
$6.44$ & $5.77$ & \textbf{2.77} & \textbf{0.00} & \textbf{0.00} & \textbf{0.00} & $1.05$ & $0.01$
\\

&
&
$0.05$
&
$6.31$ & $5.40$ & \textbf{2.66} & \textbf{0.00} & \textbf{0.00} & \textbf{0.00} & $1.20$ & $0.05$
\\

\cmidrule(lr){1-11}


\multirow{4}{*}{$(200,400)$}
&
\multirow{2}{*}{$0.05$}
&
$0.01$
&
$56.05$ & $456.11$ & \textbf{21.12} & \textbf{0.00} & \textbf{0.00} & \textbf{0.00} & $1.12$ & $0.01$
\\

&
&
$0.05$
&
$165.00$ & $136.71$ & \textbf{40.56} & \textbf{0.00} & \textbf{0.00} & \textbf{0.00} & $1.51$ & $0.05$
\\

\cmidrule(lr){2-11}

&
\multirow{2}{*}{$0.10$}
&
$0.01$
&
$19.41$ & \textbf{18.45} & $24.59$ & \textbf{0.00} & \textbf{0.00} & \textbf{0.00} & $1.07$ & $0.01$
\\

&
&
$0.05$
&
$13.47$ & $18.70$ & \textbf{9.72} & \textbf{0.00} & \textbf{0.00} & \textbf{0.00} & $1.34$ & $0.05$
\\

\cmidrule(lr){1-11}


\multirow{4}{*}{$(200,600)$}
&
\multirow{2}{*}{$0.05$}
&
$0.01$
&
$29.03$ & \textbf{20.83} & $32.27$ & \textbf{0.00} & \textbf{0.00} & \textbf{0.00} & $1.08$ & $0.01$
\\

&
&
$0.05$
&
$20.47$ & $21.33$ & \textbf{15.48} & \textbf{0.00} & \textbf{0.00} & \textbf{0.00} & $1.30$ & $0.05$
\\

\cmidrule(lr){2-11}

&
\multirow{2}{*}{$0.10$}
&
$0.01$
&
$24.04$ & $16.57$ & \textbf{13.91} & \textbf{0.00} & \textbf{0.00} & \textbf{0.00} & $1.05$ & $0.01$
\\

&
&
$0.05$
&
$26.54$ & $29.03$ & \textbf{20.94} & \textbf{0.00} & \textbf{0.00} & \textbf{0.00} & $1.20$ & $0.05$
\\

\cmidrule(lr){1-11}


\multirow{4}{*}{$(300,600)$}
&
\multirow{2}{*}{$0.05$}
&
$0.01$
&
$1280.86$ & \textbf{1229.48} & $1233.51$ & $0.26$ & $0.28$ & \textbf{0.23} & $1.13$ & $0.01$
\\

&
&
$0.05$
&
$1255.78$ & \textbf{1231.19} & $1281.97$ & $0.18$ & $0.16$ & \textbf{0.13} & $1.52$ & $0.05$
\\

\cmidrule(lr){2-11}

&
\multirow{2}{*}{$0.10$}
&
$0.01$
&
$1271.12$ & \textbf{1270.88} & $1298.27$ & \textbf{0.10} & $0.12$ & $0.13$ & $1.07$ &$ 0.01$
\\

&
&
$0.05$
&
$1241.59$ & $1249.87$ & \textbf{1222.78} & $0.11$ & $0.10$ & \textbf{0.04} & $1.34$ & $0.05$
\\

\cmidrule(lr){1-11}


\multirow{4}{*}{$(300,900)$}
&
\multirow{2}{*}{$0.05$}
&
$0.01$
&
\textbf{728.03} & $598.70$ & $1249.15$ & \textbf{0.00} & \textbf{0.00} & $0.07$ & $1.08$ & $0.01$
\\

&
&
$0.05$
&
$1278.88$ & \textbf{667.94} & $1542.93$ & $0.06$ & \textbf{0.00} & $0.03$ &$1.30$& $0.05$
\\

\cmidrule(lr){2-11}

&
\multirow{2}{*}{$0.10$}
&
$0.01$
&
$954.43$ & \textbf{947.79} & $2410.14$ & \textbf{0.00} & \textbf{0.00} & $0.04$ & $1.05$ & $0.01$
\\

&
&
$0.05$
&
$1586.83$ & $2327.39$ & \textbf{1526.27} & \textbf{0.11} & \textbf{0.11} & $0.13$ & $1.20$ &$ 0.05$
\\

\cmidrule(lr){1-11}


\multirow{4}{*}{$(400,800)$}
&
\multirow{2}{*}{$0.05$}
&
$0.01$
&
\textbf{2415.06} & $2418.65$ & $2420.99$ & $0.31$ & $0.42$ & \textbf{0.25} & $1.13$ & $0.01$
\\

&
&
$0.05$
&
\textbf{2414.38} & $2414.66$ & $2418.10$ & \textbf{0.17} & $0.18$ & $0.22$ & $1.52$& $0.05$
\\

\cmidrule(lr){2-11}

&
\multirow{2}{*}{$0.10$}
&
$0.01$
&
$2420.19$ & \textbf{2415.88} & $2417.56$ & $0.22$ & $0.22$ & \textbf{0.20} & $1.07$ & $0.01$
\\

&
&
$0.05$
&
$2440.07$ & \textbf{2423.58} & $2457.53$ & $0.21$ & $0.21$ & \textbf{0.15} & $1.34$ & $0.05$
\\

\cmidrule(lr){1-11}


\multirow{4}{*}{$(400,1200)$}
&
\multirow{2}{*}{$0.05$}
&
$0.01$
&
$1269.17$ & $1308.03$ & \textbf{1254.63} & $0.23$ & $0.24$ & \textbf{0.18} & $1.08$ & $0.01$
\\

&
&
$0.05$
&
$1278.52$ & \textbf{1268.26} & $1309.52$ & $0.12$ & $0.12$ & \textbf{0.07} & $1.30$ & $0.05$
\\

\cmidrule(lr){2-11}

&
\multirow{2}{*}{$0.10$}
&
$0.01$
&
$1315.15$ & $1320.23$ & \textbf{1315.03} & $0.08$ & $0.08$ & \textbf{0.07} & $1.05$ & $0.01$
\\

&
&
$0.05$
&
$1414.05$ & $1427.79$ & \textbf{1370.80} & $0.10$ & \textbf{0.07} & \textbf{0.07} & $1.20$ & $0.05$
\\

\cmidrule(lr){1-11}


\multirow{4}{*}{$(500,1000)$}
&
\multirow{2}{*}{$0.05$}
&
$0.01$
&
$2800.62$ & \textbf{2698.31} & $3600.00$ & $0.76$ & $0.67$ & \textbf{0.56} & $1.13$ & $0.01$
\\

&
&
$0.05$
&
$3600.00$ & \textbf{3355.60} & $3600.00$ & $0.50$ & $0.52$ & \textbf{0.47} & $1.53$ & $0.05$
\\

\cmidrule(lr){2-11}

&
\multirow{2}{*}{$0.10$}
&
$0.01$
&
\textbf{3600.00} & \textbf{3600.00} & \textbf{3600.00} & $0.49$ & $0.45$ & \textbf{0.41} & $1.07$ & $0.01$
\\

&
&
$0.05$
&
\textbf{3600.00} & \textbf{3600.00} & \textbf{3600.00} & $0.44$ & $0.46$ & \textbf{0.37} & $1.34$ & $0.05$
\\

\cmidrule(lr){1-11}


\multirow{4}{*}{$(500,1500)$}
&
\multirow{2}{*}{$0.05$}
&
$0.01$
&
$2587.44$ & $2518.81$ & \textbf{2502.63} & \textbf{0.24} & $0.29$ & $0.28$ & $1.08$ & $0.01$
\\

&
&
$0.05$
&
\textbf{2720.37} & $2958.32$ & $3600.00$ & $0.22$ & $0.22$ & \textbf{0.21} & $1.30$ & $0.05$
\\

\cmidrule(lr){2-11}

&
\multirow{2}{*}{$0.10$}
&
$0.01$
&
$3600.00$ & \textbf{3474.86} & $3600.00$ & $0.20$ & $0.20$ & \textbf{0.19} & $1.05$ & $0.01$
\\

&
&
$0.05$
&
$3600.00$ & $3600.00$ & \textbf{3073.37} & $0.17$ & $0.20$ & \textbf{0.14} & $1.20$ & $0.05$
\\

\bottomrule
\end{tabular}%
}
}
{}\begingroup\fontsize{8}{12}\selectfont\emph{Notes.}
The ``Opt. ratio'' column reports the average ratio of the objective value obtained by \emph{UB-Incl-Box} to the corresponding incumbent objective value obtained by the exact formulations, \emph{UB-Lin} and \emph{UB-Incl}. The $\lambda_{\mathcal U}$ column measures the potential approximation error induced by the box approximation.
\par
\endgroup
\end{table}


\newpage
\section{Conclusion}
\label{sec:conclusion}
This paper studies robust network enhancement through lower- and upper-bound formulations under uncertain network utilities and enhancement costs. We develop exact reformulations for both problems and examine their computational properties. For the lower-bound problem, we identify the robust separation subproblem as a key computational challenge and characterize the emergence of fractional extreme points in its linearized formulation. For the upper-bound problem, we derive two exact reformulations and compare their linear relaxation strength.

To improve computational tractability, we develop a box approximation of the utility uncertainty set that preserves robust feasibility and provides explicit performance guarantees for both problems. Computational experiments on spanning tree enhancement show that the approximation improves computational efficiency while maintaining solution quality close to the exact formulations. These results illustrate how uncertainty-set approximation can balance computational tractability and solution quality in robust network enhancement.

One direction for future exploration is to characterize robust counterparts of facet-defining inequalities from classical network optimization, with the aim of developing stronger formulations that exploit the underlying polyhedral structure under uncertainty.


\newpage
\bibliographystyle{plainnat}
\bibliography{bibi/references}

\end{document}